\documentclass[12pt]{article}
\usepackage[margin=1in]{geometry} 
\usepackage{amsmath,amsthm,amssymb}
\usepackage{cancel}
\usepackage{xcolor}
\usepackage[english]{babel}
\usepackage[utf8x]{inputenc}
\usepackage{amsmath}
\usepackage{amssymb}
\usepackage{amsthm}
\usepackage{graphicx}

\usepackage{enumitem}
\usepackage{listings}
\usepackage{geometry}
\usepackage[font=small,labelfont=bf]{caption}
\newcommand{\R}{\mathbb{R}}
\newcommand{\Z}{\mathbb{Z}}
\newcommand{\N}{\mathbb{N}}

\newcommand{\T}{\mathbb{T}}

\newcommand{\bB}{\mathcal{B}}
\newcommand{\cK}{\mathcal{K}}

\newcommand{\s}{\sigma}
\newcommand{\go}{\omega}
\newcommand{\gl}{\lambda}

\newcommand{\Dt}{{\Delta t}}

\newcommand{\ra}{\rho_\alpha}

\newcommand{\bc}{\overline c}

\newcommand{\tdU}{\tilde U}
\newcommand{\tdM}{\tilde M}
\newcommand{\diver}{{\rm{div}}}
\theoremstyle{definition}
\newtheorem{theorem}{Theorem}
\numberwithin{theorem}{section}
\newtheorem{lemma}[theorem]{Lemma}

\newtheorem{proposition}[theorem]{Proposition}

\newtheorem{remark}[theorem]{Remark}

\numberwithin{equation}{section}
\begin{document}
 
\renewcommand{\thefootnote}{\fnsymbol{footnote}}	
\title{Approximation of an optimal control problem for the time-fractional
Fokker-Planck equation}
\author{Fabio Camilli\footnotemark[1] \and Serikbolsyn Duisembay\footnotemark[2]\and 
Qing Tang\footnotemark[3]}
\date{\today}
\footnotetext[1]{Dip. di Scienze di Base e Applicate per l'Ingegneria,  Sapienza Universit{\`a}  di Roma, via Scarpa 16,	00161 Roma, Italy, ({\tt e-mail: fabio.camilli@sbai.uniroma1.it})}
\footnotetext[2]{ King Abdullah University of Science and Technology
	(KAUST),  CEMSE division, Thuwal 23955-6900, Saudi Arabia, ({\tt e-mail: serikbolsyn.duisembay@kaust.edu.sa})}
\footnotetext[3]{China University of Geosciences (Wuhan), Wuhan, China ({\tt e-mail: tangqingthomas@gmail.com })}
\maketitle
\begin{abstract}
	In this paper, we study the numerical approximation of a system of PDEs with fractional time derivatives.
	This system is derived from an optimal control problem for a time-fractional Fokker-Planck equation with time dependent drift by convex duality argument. The system is composed by a time-fractional  backward Hamilton-Jacobi-Bellman and a forward Fokker-Planck equation and can be used to describe the evolution of probability density of particles trapped in anomalous diffusion regimes. We approximate Caputo derivatives in the system by means of L1 schemes and the Hamiltonian by finite differences. The scheme for the Fokker-Planck equation is constructed such that the duality structure of the PDE system is preserved on the discrete level. We prove the well posedness of the scheme and the convergence to the solution of the continuous problem.	
\end{abstract}
\begin{description}
	\item [{\bf AMS subject classification}:] 65N06, 91A13, 35R11 .
	\item[{\bf Keywords}:]   Mean Field Games, time-fractional Fokker-Planck equation, Caputo   derivative,   finite differences, convergence.
\end{description}
\section{Introduction}
In this paper, we  study the numerical approximation of the  system
\begin{equation}\label{eq:frac_mfg}
\begin{cases}
\partial_{[t,T)}^\alpha u - \Delta u + H(x,Du) = f[m](x) &(x,t)\in\T^2\times (0,T),\\[4pt]
\partial_{(0,t]}^\alpha m - \Delta m - \diver(mH_p(x,Du)) = 0&(x,t)\in\T^2\times (0,T),\\[4pt]
u(x,T)=u_T(x),\,m(x,0)=m_0(x), \\[4pt]
\int_{\T^2}m_0(x)dx=1,
\end{cases}
\end{equation}
where $\T^2$ denotes the unit torus in $\R^2$ and   $\partial_{(0,t]}^\alpha\cdot$, $\partial_{[t,T)}^\alpha\cdot$  are the forward and backward time-fractional  Caputo derivatives defined, for $\alpha\in (0,1)$, by 
\begin{align}
&\partial_{(0,t]}^\alpha m(x,t) = \frac{1}{\Gamma(1-\alpha)}\int_{0}^{t}\frac{\partial_s m(x,s)}{(t-s)^{\alpha}}ds, \label{eq:caputo_forward}\\
&\partial_{[t,T)}^\alpha u(x,t) = -\frac{1}{\Gamma(1-\alpha)}\int_{t}^{T}\frac{\partial_su(x,s)}{(s-t)^{\alpha}}ds.\label{eq:caputo_backward}
\end{align}
The system is composed by time-fractional backward  Hamilton-Jacobi-Bellman   and 
forward Fokker-Planck equations with the corresponding final and initial conditions.
If the coupling cost $f[m]$ admits a primitive $F(x,m)$,  \eqref{eq:frac_mfg} can be derived via convex duality from an optimal control problem of 
\[\inf_{v,m}\int_0^t\int_{\T^2} (mL(x,-v)+F(x,m))dx\,ds+\int_{\T^2}(I^{1-\alpha}_{(0,t]}m)(T)dx,\]
 subject to the constraint given by the time-fractional Fokker-Planck equation
\begin{equation}\label{frac_fp}
\begin{cases}
\partial_{(0,t]}^\alpha m - \Delta m - \diver(mv) = 0,&(x,t)\in\T^2\times (0,T),\\[4pt]
m(x,0)=m_0(x),\\[4pt]
\int_{\T^2}m_0(x)dx=1.
\end{cases}
\end{equation}
Here the forward Riemann-Liouville fractional integral is defined as 
$$
I^{1-\alpha}_{(0,t]}m(t)=\frac{1}{\Gamma(1-\alpha)}\int_0^t\frac{m(s)}{(t-s)^\alpha}ds.
$$
System \eqref{eq:frac_mfg} is derived in  \cite{q}, where the Fokker-Planck equation \eqref{frac_fp} is considered in its weak formulation via a general fractional Sobolev space framework. The wellposedness of the coupled system is obtained using compactness and Schauder fixed-point argument. Time fractional advection-diffusion equations have been extensively used for studying evolution of probability of particles governed by subdiffusion behavior. This refers to the trajectory of a single particle is modeled by non Markovian stochastic process where the mean square displacement is no longer linear \cite{abmw, bg,mk,ms,zaslavsky2002}. The nonlocal structure of fractional time derivative makes it very useful for describing memory effects and deriving the power law decay structure of energy in a diffusion-transport PDE system \cite{du,jlz,Jin2019,tt}.  \\
\indent Time fractional Hamilton-Jacobi-Bellman equation has been considered from a probabilistic point of view by Kolokoltsov et al.\cite{Kolokoltsov2014}.  The theory of viscosity solution of time-fractional Hamilton-Jacobi-Bellman equation has been constructed by Topp and Yangari \cite{Topp2017}, Giga, Namba et al. \cite{glm,Namba2018}. Recently Camilli and Goffi studied   weak solutions to time-fractional Hamilton-Jacobi-Bellman equation in a Sobolev space \cite{cg}. 
The fractional Fokker-Planck equation can be more suitably considered in the variational formulation and the weak solution is in the sense of distributions. The natural functional structure would be that of vector valued fractional Sobolev space and wellposedness is usually established by a priori estimates based on energy dissipation. This approach has been adopted in a large number of works, \cite{du,du2018invitation,Gorenflo2015,Jin_2015,Jin2019,Shen_2019,q,tt} for examples. \\
\indent From a numerical point view, the essential step to study the fractional PDE is to approximate the time-fractional derivative operator. We use a  classical scheme called L1 approximation. This scheme is naturally derived from the approximation of the fractional integral as a Riemann sum and long known to be consistent (for an concise illustration we refer to section 3 of \cite{tt}). Moreover, L1 approximation scheme stands out by being able to preserve at the discrete level certain desirable features of the origin PDEs, such as maximum principle \cite{glm} and energy stability \cite{du,Shen_2019,tt}. Therefore, approximations can be constructed to converge to suitable notions of weak solutions, which may be viscosity solutions to Hamilton-Jacobi-Bellman equations \cite{glm}, or distributional solutions in Sobolev spaces to diffusion (phase field) equations \cite{du,Shen_2019,tt}. \\
\indent If fractional derivatives are simply replaced by first order derivatives, system \eqref{eq:frac_mfg} becomes the classical Mean Field Games system introduced in \cite{ll} to characterize the Nash equilibria for differential games with a large number of agents. Achdou et al. \cite{acd,acc} introduced an finite difference type discretization of mean field games system where the scheme for the Fokker-Planck equation is deduced via a duality argument, reproducing at a discrete level a crucial property of the continuous system. This type of schemes has been further developed in \cite{afg,bks,cs}. \\
\indent In this paper, we present a finite differences approximation of \eqref{eq:frac_mfg} by combining the classical approximation scheme for mean field games system developed in \cite{acd,acc} with L1 schemes for the fractional derivatives. The duality structure of the time-fractional PDE system and the energy stable structure of the advection diffusion equation are shown to be preserved. Under appropriate monotonicity and convexity assumptions, we prove existence, uniqueness of the solution to the discrete system and convergence to the solution of the original PDE system. Some numerical tests are performed to illustrate the different behavior of the time-fractional case with respect to the standard one.\\
\indent The paper is organized as follows. In Section \ref{sec:scheme}, we introduce the finite differences scheme and prove some preliminary results.  Section \ref{sec:well_posed} is devoted to show the well posedness of the discrete problem. In Section \ref{sec:convergence}, we prove the convergence of the scheme and in Section \ref{sec:numerics} we carry out some numerical simulations. Finally we discuss some possible directions for future works. 


\section{The numerical scheme}\label{sec:scheme}
In this section we introduce a  numerical scheme for \eqref{eq:frac_mfg} and  we prove some   preliminary results  to show  its well-posedness.\\
\indent For simplicity of notations, we assume that the dimension $d$ is equal to $2$, since the extension for general $d$ will be clear from this special case.
Let $\T^2_h$ be a uniform grid on the torus with step $h$, (this supposes that $N_h = 1/{h}$ is an integer), and denote by $x_{i,j}$   a generic point in  $\T^2_h$ (an anisotropic mesh with steps $h_1$ and $h_2$ is possible too and we have taken $h_1=h_2$ only for simplicity). The values  $U^n_{i,j}$ and $M^n_{i,j}$ denotes the numerical approximation of the function $u$ and $m$ at $(x_{i,j},t_n)=(ih,jh, n\Dt)$, $0\le i,j\le N_h$, $n=0,\dots,N$ 
(assuming that $N=T/\Dt$ is an integer). We also denote by  $U^n$ and $M^n$  the grid function taking the values $U^n_{i,j}$ and $M^n_{i,j}$ at $x_{i,j}\in \T^2_h$. For a generic pair of grid functions $U$, $V$ on $\T^2_h$, we  define the scalar product \\
$$
  (U,V)_2=h^2\sum_{  i,j } U_{i,j} V_{i,j},
  $$
 the norm\\
 $$
  \|U\|_2^2=(U,U)_2
  $$
   and the norm 
  $$ 
   \|U\|_\infty=\max_{0\le i,j <N_h}|U_{ij}|.
   $$
    We consider  the compact and convex set
\[
\cK_h=\{M= (M_{i,j})_{ 0\le i,j <N_h}:   (M,\overline 1)_2=1;\quad   M_{i,j}\ge 0 \}\,,
\]
where $\overline 1$  denotes the $N_h^2$-tuple with all components equal to $1$. Note that $\cK_h$  can be viewed as the set of the discrete probability measures on $\T^2_h$. \\
\indent We first recall the L1   numerical method for the discretization of the forward and backward Caputo time-fractional derivatives (see \cite{jlz,lx}). 
The approximations of \eqref{eq:caputo_forward} and \eqref{eq:caputo_backward} at $(x_{i,j},t_n)$ are given by
\begin{align}
D_{\Delta t}^{\alpha} M_{i,j}^{n} = \frac{1}{\rho_\alpha} \left( M_{i,j}^{n} - \sum_{k=0}^{n-1} c_k^{n} M_{i,j}^{k} \right),\label{eq:discr_frac_deriv_forw}\\
\bar D_{\Delta t}^{\alpha} U_{i,j}^{n} = \frac{1}{\rho_\alpha}\left( U_{i,j}^{n} - \sum_{k=n+1}^{N} \bc_{n}^k U_{i,j}^k\right) ,\label{eq:discr_frac_deriv_back}
\end{align}
where 
$$
\rho_\alpha = \Gamma(2-\alpha)\Delta t^{\alpha}
$$
and the forward and backward coefficients are defined by
\[
c_k^n = 
\begin{cases}
n^{1-\alpha}-(n-1)^{1-\alpha}, \quad k=0,\\
2(n-k)^{1-\alpha} - (n+1-k)^{1-\alpha} - (n-1-k)^{1-\alpha}, \quad k=1,\dots, n-1,
\end{cases}
\]
and
\[
\bc_n^k = 
\begin{cases}
2(k-n)^{1-\alpha} - (k+1-n)^{1-\alpha} - (k-1-n)^{1-\alpha}, \quad k=n+1,\dots, N-1, \\
(N-n)^{1-\alpha}-(N-1-n)^{1-\alpha}, \quad k=N.
\end{cases}
\]
In the following lemma, we summarize some properties of the coefficients  $c_k^n$, $\bc_n^k$.
\begin{lemma}\label{lem:c_properties}
We have
	\begin{enumerate}
		\item[(i)] $c_k^{n} >0$ for $0\le k \leq n-1$ and  $\bc_n^m >0$ for $ n+1\le m\le N$.
		\item[(ii)] $\bar c_k^n = c_{k+1}^{n+1}$ for $1\le n\le N-1$ and $0\le k\le n-1$.
		\item[(iii)] $\sum_{k=0}^{n-1} c_k^{n}=1$ 
		and $\sum_{k=n+1}^{N} \bc_n^k=1.$
	\end{enumerate}
\end{lemma}
\indent By simple calculation, it can be shown that an equivalent form of the L1 scheme of the discrete forward Caputo derivative is
\begin{equation}\label{discreteCaputo2}
D_{\Delta t}^{\alpha} U_{i,j}^{n} =\sum_{k=0}^nb_k\frac{U_{i,j}^{n-k}-U_{i,j}^{n-1-k}}{\Delta t},
\end{equation}
where 
$$
b_k=\frac{\Delta t^{1-\alpha}}{\Gamma(2-\alpha)}[(k+1)^{1-\alpha}-k^{1-\alpha}].
$$
\indent When written in this form and used for fractional phase field equations it has been shown the L1 scheme of Caputo derivative can preserve on a discrete level maximum principle, energy stability and mass conservation \cite{tt}.\\
\indent For the discretization  in space, we rely on the approach  in \cite{acd}.  We introduce the finite differences operators
\begin{equation}
\label{eq:finite_diff}
(D_1^+ U )_{i,j} = \frac{ U_{i+1,j}-U_{i,j}   } {h} \quad \hbox{and }\quad  (D_2^+ U )_{i,j} = \frac{ U_{i,j+1}-U_{i,j}   } {h},
\end{equation}
and define
\[
[D_h U]_{i,j} =\left((D_1^+ U )_{i,j} , (D_1^+ U )_{i-1,j}, (D_2^+ U )_{i,j}, (D_2^+ U )_{i,j-1}\right) ^T.\label{eq:discrete_gradient}
\]
\indent For the  discretization of the Laplace operator, we consider a standard five points approximation
\[
(\Delta_h U)_{i,j} = -\frac{1}{h^2} (4 U_{i,j} - U_{i+1,j} - U_{i-1,j} - U_{i,j+1} - U_{i,j-1}).
\]
\indent In order to approximate the Hamiltonian $H$ in equation \eqref{eq:frac_mfg}, we consider a  numerical Hamiltonian $g: \T^2 \times \R^4\to \R$,  $(x,q_1,q_2,q_3,q_4)\mapsto g\left(x,q_1,q_2,q_3,q_4\right)$
satisfying the following conditions:
\begin{itemize}
	\item[(G1)] $g$ is nonincreasing with respect to $q_1$ and $q_3$, and nondecreasing with respect to $q_2$ and $q_4$.
	\item[(G2)] $g$ is consistent with the Hamiltonian $H$, i.e. 
	\begin{displaymath}
	g(x,q_1,q_1,q_2,q_2)=H(x,q), \quad \forall x\in \T^2,   q=(q_1,q_2)\in \R^2.
	\end{displaymath}
	\item[(G3)]  $g$ is of class $C^1$.
	\item[(G4)] the function $(q_1,q_2,q_3,q_4) \mapsto g(x,q_1,q_2,q_3,q_4)$ is convex.
\end{itemize}

\indent For the discretization of the coupling term $f[m]$ in the Hamilton-Jacobi-Bellman equation, denoted by  $m_h$   the piecewise constant function with   value $M_{i,j}$ in the square $|x-x_{i,j}|_\infty \le h/2$, we  assume that $f[m_h]$ can be calculated in practice. Then, the approximation of $f[m](x_{i,j})$ is given by
$(f_h[M])_{i,j} = f[m_h](x_{i,j})$. In particular, if $f$ is a local operator, i.e. $f[m](x)=f(m(x))$, then we set $f_h[M]_{j,k}=f(M_{j,k})$.\\
\indent We assume that:
\begin{itemize}
	\item[(F1)]  $f_h$ is continuous and maps   $\cK_h$  on a bounded set of grid functions.
	\item[(F2)] $f_h$ is monotone, i.e. for $M,\bar M\in \cK_h$
		\[
		\left( f_h[M]-f_h[\bar M],M-\bar  M \right)_2 \le 0 \Rightarrow  M=  \bar M.
		\]
	\item[(F3)] There exists $C$ independent of $h$ such that
	for all grid functions  $M\in \cK_h$
	\begin{align*}
	& \| f_h[M] \|_{\infty} \le C\\
	& | f_h[M]_{i,j} - f_h[M]_{\bar i,\bar j} | \le C |x_{i,j}-x_{\bar i,\bar j}|\qquad \forall 0\le i,j,\bar i,\bar j \le N_h.
	\end{align*}
	\item[(F4)]   For any $m\in \cK=\{m\in L^1(\T^2): m\ge 0, \int_{\T^2}m(x)dx=1\}$, defined 
	\[M_h=\frac 1 {h^2} \int_{|x-x_{i,j}|_{\infty}\le h/2} m(x)dx,\] then 
		\[
		\lim_{h\to 0}\,\sup_{m\in\cK}\|f[m]-f_h[M_h]\|_\infty =0.
		\]
\end{itemize}
\indent We approximate the Hamilton-Jacobi-Bellman equation in \eqref{eq:frac_mfg} by means of the scheme
\begin{equation}\label{discrete_HJequation}
\left\{
\begin{array}{ll}
\bar D_{\Delta t}^{\alpha} U^n_{i,j} - (\Delta_h U^{n})_{i,j} + g(x_{i,j},[D_h U^{n}]_{i,j}) = f_h[M^{n+1}]_{i,j}, \quad& n=0,\dots, N-1,\, 0\le i,j\le N_h,\\[6pt]
U^N_{i,j}=u_T(x_{i,j})& 0\le i,j\le N_h,
\end{array}
\right.
\end{equation}
for a given  $M= (M^n)_{ 0\le n\le N-1}$. Note that the previous scheme is backward, implicit and non local in the time variable,  since the solution at step $n$ depends on the values of the solution computed at previous steps $n+1,\dots, N$.\\
\indent Concerning the approximation of the Fokker-Planck equation, following \cite{acd}, we note that the weak formulation of previous equation involves the identity
\[ -\int_{\T^2} \diver\left(m   H_p(x,D u)\right) w\, dx= \int_{\T^2} m H_p(x,D u) \cdot D w\, dx,\]
for any test function $w$. \\
\indent Using discrete integration by parts, it can be approximated by
\begin{equation}\label{integralterm}
-h^2 \sum_{i,j} \bB_{i,j}(U,M) W_{i,j}= h^2\sum_{i,j}    M_{i,j} g_q (x_{i,j},[D_h U]_{i,j}) \cdot [D_h W]_{i,j},
\end{equation}
where
\[
\bB_{i,j}(U,M) = 
{\small \frac{1}{h} \left(
\begin{aligned}
&  \left( \begin{aligned} & M_{i,j} \frac{\partial g}{\partial q_1}(x_{i,j}, [D_hU]_{i,j}) - M_{i-1,j} \frac{\partial g}{\partial q_1}(x_{i-1,j}, [D_hU]_{i-1,j}) \\
& + M_{i+1,j} \frac{\partial g}{\partial q_2}(x_{i+1,j}, [D_hU]_{i+1,j}) - M_{i,j} \frac{\partial g}{\partial q_2}(x_{i,j}, [D_hU]_{i,j})  \end{aligned}\right) \\
& + \left( \begin{aligned} & M_{i,j} \frac{\partial g}{\partial q_3}(x_{i,j}, [D_hU]_{i,j}) - M_{i,j-1} \frac{\partial g}{\partial q_3}(x_{1,j-1}, [D_hU]_{1,j-1}) \\
& + M_{i,j+1} \frac{\partial g}{\partial q_4}(x_{i,j+1}, [D_hU]_{i,j+1}) - M_{i,j} \frac{\partial g}{\partial q_4}(x_{i,j}, [D_hU]_{i,j})  \end{aligned}\right)
\end{aligned}\right).}
\]
\indent This yields the following scheme:
\begin{equation}\label{discrete_FPequation}
\left\{
\begin{array}{ll}
D_{\Delta t}^{\alpha} M^{n+1}_{i,j}-(\Delta_h M^{n+1})_{i,j} - \bB_{i,j}(U^{n},M^{n+1})= 0, \quad &n=0,\dots, N-1,\, 0\le i,j\le N_h,\\[6pt]
M_{i,j}^{0}=  \frac 1 {h^2} \int_{|x-x_{i,j}|_{\infty}\le h/2} m_0(x) dx&0\le i,j\le N_h.\\
\end{array}
\right.
\end{equation}
for a given vector $U=  (U^n_{i,j})_{ 1\le n\le N}$. \\
Collecting \eqref{discrete_HJequation} and\eqref{discrete_FPequation}, we end up with the following scheme for \eqref{eq:frac_mfg}
\begin{equation}\label{eq:scheme}
\left\{
 \begin{array}{ll}
\bar D_{\Delta t}^{\alpha} U_{i,j}^{n} - (\Delta_h U^{n})_{i,j} + g(x_{i,j},[D_hU^{n}]_{i,j}) = f_h[M^{n+1}]_{i,j} ,\\[6pt]
D_{\Delta t}^\alpha M_{i,j}^{n+1} - (\Delta_h M^{n+1})_{i,j} - \bB_{i,j}(U^{n},M^{n+1}) = 0 ,\\[6pt]
U^{N_T}_{i,j} = u_T(x_{i,j}), \,   M_{i,j}^{0}=  \frac 1 {h^2} \int_{|x-x_{i,j}|_{\infty}\le h/2} m_0(x) dx,
\end{array}
\right.
\end{equation}
for $n=0,\dots,N-1$ and $0\le i,j\le N_h$.\\
\indent We prove  some results  for the problems \eqref{discrete_HJequation} and \eqref{discrete_FPequation} which will be used in the following. 
\begin{lemma}
\label{le:existence} Let $V$ be a grid function on $\T_h^2$ and $\delta$ be a positive parameter. Then, there exists a unique grid function $U$ such that
\begin{equation}\label{eq:stationary}
\delta U_{i,j} - (\Delta_hU)_{i,j} + g(x_{i,j},[D_hU]_{i,j}) = V_{i,j}.
\end{equation}
\end{lemma}
For the proof of the previous result, we refer  to  \cite[Lemma 1]{acd}.
\begin{proposition}\label{prop:ex_HJ}
There exists a  function $U=(U^n)_{n=0,\dots,N}$  which satisfies \eqref{discrete_HJequation}. Moreover, if $U= (U^n)_{ 0\le n\le N }$, $V= (U^n)_{ 0\le n\le N }$ are such that $U^N\le V^N$ and 
\begin{align*}
\bar D_{\Delta t}^{\alpha} U_{i,j}^{n} - (\Delta_h U^{n})_{i,j} + g(x_{i,j},[D_hU^{n}]_{i,j}) \le f_h[M^{n+1}]_{i,j},\quad n=0,\dots,N-1,\\
\bar D_{\Delta t}^{\alpha} V_{i,j}^{n} - (\Delta_h V^{n})_{i,j} + g(x_{i,j},[D_hV^{n}]_{i,j}) \ge  f_h[M^{n+1}]_{i,j},\quad n=0,\dots,N-1,
\end{align*}
then $U^n\le V^n$ for any $n=0,\dots,N$.
\end{proposition}
\begin{proof}
The existence and uniqueness of $U^{n}$, $n=0, \dots, N-1$ is obtained by induction: at each step, we use Lemma \ref{le:existence} with
\[
\delta = \frac{1}{\rho_\alpha},\qquad V_{i,j} =  \frac{1}{\rho_\alpha} \sum_{k=n+1}^{N} \bar c_{n}^k U_{i,j}^k + f_h[M^{n+1}]_{i,j}.
\]
\indent The second part of the statement follows immediately by a comparison principle for \eqref{eq:stationary}, which relies on the monotonicity of $g$  and the positivity of the coefficients $\bc_n^k$. 
\end{proof}
To prove that a solution of \eqref{discrete_HJequation} is uniformly Lipschitz continuous, we need the following  fractional version of the discrete Gronwall Lemma (see \cite[Lemma 3.1]{lwz}).
\begin{lemma}
	Suppose that the nonnegative sequences $\{\go^n,g^n\}_{n\in\N}$ satisfy  $D_{\Delta t}^\alpha\go^1\le \gl_1 \go^1+\gl_2 \go^0+g^1$ and
	\[D_{\Delta t}^\alpha \go^n\le \gl_1\go^n+\gl_2\go^{n-1}+\gl_3\go^{n-2}+g^n,\quad n\ge 2 \]
where $\gl_1,\gl_2,\gl_3$ are given nonnegative constants independent of $\Dt$. Then, there exists $\delta>0$ such for $\Dt<\delta$,
\begin{equation}\label{eq:discrete_Gronwall}
\go^n\le \left(\go^0+\frac{(n\Dt)^\alpha}{\Gamma(1+\alpha)}\max_{1\le j \le n}g^j\right)E_\alpha(2\lambda (n\Dt)^\alpha ),\qquad 1\le n\le N,
\end{equation}
where $E_\alpha(z)=\sum_{k=0}^\infty z^k/\Gamma(1+k\alpha)$ is the Mittag-Leffler function and $\gl=\gl_1+\gl_2/(2-2^{1-\alpha})+\gl_3/(2\cdot2^{1-\alpha}-1^{1-\alpha}-3^{1-\alpha})$.
\end{lemma}

\begin{proposition}\label{prop:lip_U}
	 Let $U=(U^n)_{n=0,\dots,N}$ be the solution of \eqref{discrete_HJequation}. Then, for $\Dt$ sufficiently small, there exists a positive constant $L$ such that
	\begin{equation}\label{eq:lip_est}
	\|D_hU^n\|_\infty \le L, \qquad  n=0,\dots,N.
	\end{equation}
\end{proposition}
\begin{proof}
	Rewrite \eqref{discrete_HJequation} as the stationary equation
\begin{equation*} 
	 U_{i,j}^{n} - \rho_\alpha (\Delta_h U^{n})_{i,j} + \rho_\alpha g(x_{i,j},[D_hU^{n}]_{i,j}) - \rho_\alpha f_h[M^{n+1}]_{i,j}= \sum_{k=n+1}^{N} \bc_{n}^k U_{i,j}^k
	 \end{equation*} 
and let  $\Psi$ be the map $U_{i,j}^n=\Psi( \sum_{k=n+1}^{N} \bc_{n}^k U_{i,j}^k)$ that associates to right hand side of the previous equation the corresponding solution. For $V \in \R^{N_h}$ and $\lambda\in\R$, we have $\Psi( V+\lambda)=\Psi( V)+\lambda$
(where, on the left side of the previous identity, $\lambda$ is interpreted as the constant function on $\R^{N_h}$)
and, by standard monotonicity argument,
\begin{equation}\label{eq:lip0}
\|\Psi(V)-\Psi(W)\|_\infty\le  \|V-W\|_\infty.
\end{equation}
for any $V, W\in \R^{N_h}$. For fixed $l,\kappa \in \Z^2$, call $\tau U$ the discrete function defined by
\[
(\tau U)^n_{i,j} = U^n_{l+i,\kappa+j}.
\]
We have that
\[ 
\bar D_{\Delta t}^{\alpha} (\tau  U)_{i,j}^{n} - (\Delta_h (\tau U)^{n})_{i,j} + g(x_{i,j},[D_h(\tau U)^{n}]_{i,j})= f_h[M^{n+1}]_{i,j}+ E_{i,j} 
\]
where
\begin{align*}
E_{i,j}& =   f_h[M^{n+1}]_{i+l,j+\kappa} -f_h[M^{n+1}]_{i,j}  - g(x_{i+l,j+\kappa},[D_h(\tau U)^{n}]_{i,j}) + g(x_{i,j},[D_h(\tau U)^{n}]_{i,j}).
\end{align*} 
Hence $ (\tau U)^n=  \Psi\left(\sum_{k=n+1}^{N} \bc_{n}^k (\tau U)^k+\rho_\alpha E\right)$. Moreover, since 
\begin{align*}
&f_h[M^{n+1}]_{i+l,j+\kappa} -f_h[M^{n+1}]_{i,j} \le C h \sqrt{l^2+\kappa^2},\\
&g(x_{i,j},[D_h(\tau U)^{n}]_{i,j}) - g(x_{i+l,j+\kappa},[D_h(\tau U)^{n}]_{i,j}) \le C(1+\|D_hU^n\|_\infty)h\sqrt{l^2+\kappa^2}
\end{align*}
it follows that
\[
\|E\|_\infty \le C(1+\|D_hU^n\|_\infty)h\sqrt{l^2+\kappa^2}.
\]
By \eqref{eq:lip0}, we have
\begin{align*}
&\|(\tau U)^n-U^n\|_\infty= \left\|\Psi\left(\sum_{k=n+1}^{N} \bc_{n}^k (\tau U)^k+\rho_\alpha E\right)
-\Psi\left(\sum_{k=n+1}^{N} \bc_{n}^k  U^m\right)\right\|_\infty\\
&\le \left\|\sum_{k=n+1}^{N} \bc_{n}^k( (\tau U)^k - U^k)\right\|_\infty+\ra\|E\|_\infty.
\end{align*} 
Hence
\[
\left\|\frac{(\tau U)^n - U^n}{h \sqrt{l^2+\kappa^2} }\right\|_\infty \le   \sum_{k=n+1}^{N}\bc_{n}^k\left\| \frac{(\tau U)^k - U^k}{h \sqrt{l^2+\kappa^2}}\right\|_\infty + C \rho_\alpha (1+ \|D_hU^{n}\|_\infty).
\]
For the arbitrariness of $l,\kappa$,  the previous inequality implies that
\[
\bar D_{\Delta t}^{\alpha}(\|D_hU^n\|_\infty)\le C\|D_hU^n\|_\infty+C
\]
and therefore, applying  a backward version of the discrete Gronwall's inequality \eqref{eq:discrete_Gronwall} with $\go^n=\|D_hU^n\|_\infty$, $\gl_1=C$, $\gl_2=\gl_3=0$ and $g^n=C$, we obtain \eqref{eq:lip_est} for some positive constant $L$.
\end{proof}
\begin{proposition}\label{prop:ex_FP}
	For $\Dt$ sufficiently small,   there exists a unique  solution $(M^n)_{n=0}^N$  to \eqref{discrete_FPequation}. Moreover, $M^n\ge 0$ and  $(M^n,\bar 1)_2=1$ for any $n=0,\dots,N$.
\end{proposition}
\begin{proof}
	Equation \eqref{discrete_FPequation} is equivalent to
\begin{equation}\label{eq:FP_0}	
	M^{n+1}_{i,j} -\rho_\alpha(\Delta_h M^{n+1})_{i,j} - \rho_\alpha\bB_{i,j}(U^{n},M^{n+1})  =  \sum_{k=0}^{n} c_k^{n+1} M_{i,j}^{k}. 	
\end{equation}
For fixed $n$, \eqref{eq:FP_0}  is a linear problem in the unknown $M^{n+1}\in (\R^{N_h^2})$, which  can be rewritten  as
	\begin{equation}\label{eq:FP_1}
	(I+\rho_\alpha A)M^{n+1}=\sum_{k=0}^{n} c_k^{n+1} M^{k}.
	\end{equation}
where $A$ is the linear operator  given by
	$A\cdot =-  \Delta_h\cdot - \bB_{i,j}(U^{n},\cdot)$. 
We observe that, due to the monotonicity assumption (G1), the matrix corresponding to $A$ has positive diagonal entries and non positive off-diagonal entries. Hence, for $\Dt$ sufficiently small, the matrix corresponding to the linear operator $I+\rho_\alpha A$ in \eqref{eq:FP_1} is a $M$-matrix and therefore non singular. It follows that, for any $n=0,\dots,N-1$, there exists a unique solution to \eqref{eq:FP_1}. 
\indent Moreover, since $M^0\ge 0$ and the inverse of a $M$-matrix is non negative, we easily get by induction that $M^n\ge 0$ for  $n=0,\dots,N-1$. \\
\indent To show that $(M^n,\bar 1)_2=1$ for any $n=0,\dots,N$, we first recall the following identity
(see \cite[eq. (50)]{acd})
\begin{equation}\label{eq:bypart_space}
\begin{aligned}
(AW,Z)_2&=\sum_{i,j}(D^+_1W)_{i,j}(D^+_1Z)_{i,j}+\sum_{i,j}(D^+_2W)_{i,j}(D^+_2Z)_{i,j}\\
&+\sum_{i,j}W_{i,j}[D_h Z]_{i,j}\cdot D_q g(x_{i,j},[D_hU]_{i,j}).
\end{aligned}
\end{equation}
From the previous identity, it follows  that $(AW,\bar 1)_2=0$ for any grid function $W$. Hence, taking the inner product of  \eqref{eq:FP_1}  with the function $\bar 1$, we get that
\[(M^{n+1},\bar 1)_2=\sum_{k=0}^{n } c_k^{n+1} (M^{k},\bar 1)_2.\]
Recalling that $M^0\in \cK_h$ and  Lemma \ref{lem:c_properties}.(iii), by the previous formula we get by induction 
that $M^{n} \in \cK_h$ for any $n=0,\dots,N$.
\end{proof}
\begin{remark}
We can provide an alternative proof to the discrete mass conservation property $(M^n,\bar 1)_2=1$ by the following observation. Since for integrating on the torus $\T^2$ the second equation, noting 
$$
\int_{\T^2}\Delta m-\diver (mH_p(x,Du))dx=0,
$$
so that
\begin{align*}
0&=\int_{\T^2}\partial_{(0,t]}^{\alpha}m(x,t)dx=\frac{1}{\Gamma(1-\alpha)}\int_{\T^2}\int_0^t\frac{1}{(t-s)^{\alpha}}\frac{\partial m(x,s)}{\partial s}dsdx\\
&=\frac{1}{\Gamma(1-\alpha)}\int_0^t\frac{1}{(t-s)^{\alpha}}\big{(}\int_{\T^2}\frac{\partial m(x,s)}{\partial s}dx\big{)}ds,\,\,\,\,\,\forall t\in [0,T].
\end{align*}
Then
$$
\frac{d}{dt}\int_{\T^2}m(x,t)dx=\int_{\T^2}\frac{\partial}{\partial t}m(x,t)dx=0,\,\,\,\,\forall t\in [0,T].
$$
Therefore we have 
$$
\int_{\T^2}m(x,t)dx=\int_{\T^2}m(x,0)dx=1,\,\,\,\,\forall t\in [0,T].
$$
Analogously, for the discrete case we can obtain
$$
\sum_{i,j}(\Delta_h M^{n+1})_{i,j} +\bB_{i,j}(U^{n},M^{n+1})=0
$$
and
\begin{align*}
0=h^2\sum_{i,j}D^{\alpha}_{\Delta t}M_{i,j}^{n+1}=D^{\alpha}_{\Delta t}(M^{n+1},\bar{1})_2,\,\,\,0\leq n\leq N-1.
\end{align*}
By using the equivalent form of L1 scheme for forward Caputo derivative \eqref{discreteCaputo2} we have
$$
\sum_{k=0}^{n}b_k\frac{(M^{n+1-k},\bar{1})_2-(M^{n-k},\bar{1})_2}{\Delta t}=0,\,\,\,0\leq n\leq N-1.
$$
It follows easily by mathematical induction that 
$$
(M^{n},\bar{1})_2=(M^{0},\bar{1})_2=1,\,\,\,1\leq n\leq N.
$$
\end{remark}
\indent The following lemma provides an integration by parts for  \eqref{eq:discr_frac_deriv_forw}, \eqref{eq:discr_frac_deriv_back}  which is the discrete counterpart of
fractional integration by parts formula with Caputo derivative  (see \cite[Lemma 2.2]{q})
\[
\int_{0}^T \partial_{(0,t]}^{\alpha} \phi(t)\kappa(t)dt + \phi(0)(I_{[t,T)}^{1-\alpha}\kappa)(0) = \int_{0}^T \phi(t) \partial_{[t,T)}^{\alpha} \kappa(t)dt + \kappa(T)(I_{(0,t]}^{1-\alpha}\phi)(T),
\]
where $\phi(t), \kappa(t)\in C^1[0,T]$.
\lemma\label{le:by_parts_n_n_plus_1}   Given $U=\{U^n\}_{n=0}^{N}$, $M=\{M^n\}_{n=0}^{N}$, 
then 
\begin{align*}
&\sum_{n=0}^{N-1} \Big(\bar D_{\Delta t}^{\alpha} U^{n}, M^{n+1}\Big)_2 + \frac{1}{\ra}\sum_{n=0}^{N-1} \bar c_n^N \left(U^N,M^{n+1}\right)_2  \\
&= \sum_{n=0}^{N-1} (D_{\Delta t}^{\alpha} M^{n+1},U^{n})_2+ \frac{1}{\ra}\sum_{n=0}^{N-1} c_0^{n+1}\left(M^0,U^n\right)_2.
\end{align*}
For the proof,  see the  Appendix. We also recall the following  result  \cite[Lemma 2.2]{glm}, which will be used  for barrier arguments.
\begin{lemma}\label{le:barrier_argument}
	For $\Dt>0$, the following inequality holds 
	\begin{equation}\label{eq:barrier1}
	\bar D_{\Delta t}^\alpha ((N-n)\Dt)^\alpha\ge \frac{\alpha(1-\alpha)}{\Gamma (2-\alpha)},\qquad n=0,\dots,N-1.
	\end{equation}
\end{lemma}
\section{Well posedness of the discrete problem}\label{sec:well_posed}
In this section, we study the well posedness of the discrete problem \eqref{eq:scheme}.

\subsection{Uniqueness}
The uniqueness  argument is classical in Mean Field Games theory  and relies on the monotonicity of the coupling cost
$f$ and convexity of the discrete Hamiltonian $g$.
\begin{theorem}\label{thm:uniqueness}
The problem  \eqref{eq:scheme} has at most one solution.
\end{theorem}
\begin{proof}
We suppose by contradiction that there are two solutions $(U,M)$ and $(\tilde{U},\tilde{M})$ such that $U^N = \tilde{U}^N$ and $M^{0} = \tilde{M}^{0}$. Then we have
\[
\begin{aligned}
\bar D_{\Delta t}^{\alpha} (U^{n}& -\tilde{U}^{n})_{i,j} - (\Delta_h (U^{n}-\tilde{U}^{n}))_{i,j} \\
& + g(x_{i,j},[D_hU^{n}]_{i,j}) - g(x_{i,j},[D_hU^{n}]_{i,j}) -  (f_h[M^{n+1}] - f_h[\tilde{M}^{n+1}] )_{i,j}=0.
\end{aligned}
\]
Multiplying by $Z_{i,j}$ and summing over all $i,j$ leads to
\begin{equation}\label{eq:u_diff}
\begin{aligned}
(\bar D_{\Delta t}^{\alpha} & (U^{n} -\tilde{U}^{n}),Z)_2 - ((\Delta_h (U^{n}-\tilde{U}^{n})),Z)_2 \\
 & + \sum_{i,j} (g(x_{i,j},[D_hU^{n}]_{i,j}) - g(x_{i,j},[D_h\tilde{U}^{n}]_{i,j}))Z_{i,j} - (( f_h[M^{n+1}] - f_h[\tilde{M}^{n+1}]),Z )_2=0.
\end{aligned}
\end{equation}
Multiplying the discrete equation \eqref{discrete_FPequation} by $W_{i,j}$ and summing over all $i,j$ we obtain
\[
(D_{\Delta t}^\alpha M^{n+1},W)_2 - (M^{n+1}, \Delta_h W)_2 + \sum_{i,j} M_{i,j}^{n+1} [D_hW]_{i,j} \cdot \nabla_q g  (x_{i,j},[D_hU^{n}]) = 0.
\]
From this and the similar equation satisfied by $\tilde{M}^{n+1}$, we obtain
\begin{equation}\label{eq:m_diff}
\begin{aligned}
0 & = (D_{\Delta t}^\alpha (M^{n+1}  -\tilde{M}^{n+1}),W)_2 - ((M^{n+1}-\tilde{M}^{n+1}), \Delta_h W)_2 \\
& + \sum_{i,j} M_{i,j}^{n+1} [D_hW]_{i,j} \cdot \nabla_q g ( x_{i,j},[D_hU^{n}])  - \sum_{i,j} \tilde{M}_{i,j}^{n+1} [D_hW]_{i,j} \cdot \nabla_q g  (x_{i,j},[D_h\tilde{U}^{n}] ).
\end{aligned}
\end{equation}

Let $Z = M^{n+1} - \tilde{M}^{n+1}$ in \eqref{eq:u_diff}   and $W = U^{n} - \tilde{U}^{n}$ in \eqref{eq:m_diff}.
Subtracting the two equations so obtained,   summing over $n=0,1,\dots,N-1$ and recalling that, by
 Lemma \ref{le:by_parts_n_n_plus_1}, we have
\[
\sum_{n=0}^{N-1} (\bar D_{\Delta t}^{\alpha} (U^{n} -\tilde{U}^{n}),M^{n+1} - \tilde{M}^{n+1})_2 =  \sum_{n=0}^{N-1} (D_{\Delta t}^\alpha (M^{n+1}  -\tilde{M}^{n+1}),U^{n} - \tilde{U}^{n})_2,
\]
we finally get
\begin{equation}\label{eq:un_nonneg_terms}
\begin{aligned}
0 & = \sum_{n=0}^{N-1}   \bigg(f_h[M^{n+1}] - f_h[\tilde{M}^{n+1}],M^{n+1} - \tilde{M}^{n+1}\bigg)_2 \\
& +   \sum_{n=0}^{N-1}  \sum_{i,j} M_{i,j}^{n+1}\bigg( g(x_{i,j},[D_h\tilde{U}^{n}]_{i,j}) - g(x_{i,j},[D_hU^{n}]_{i,j})  \\
& -[D_h(\tilde{U}^{n}-U^{n})]_{i,j} \cdot \nabla_q g ( x_{i,j},[D_hU^{n}])   \bigg)\\
& + \sum_{n=0}^{N-1}  \sum_{i,j} \tilde{M}_{i,j}^{n+1} \bigg(g(x_{i,j},[D_hU^{n}]_{i,j}) - g(x_{i,j},[D_h\tilde{U}^{n}]_{i,j}) \\
& -[D_h(U^{n}-\tilde{U}^{n})]_{i,j} \cdot \nabla_q g ( x_{i,j},[D_h\tilde{U}^{n}] )  \bigg).
\end{aligned}
\end{equation}
By the   monotonicity of $f_h[\cdot]$ and assumption (G4), all three terms on RHS of \eqref{eq:un_nonneg_terms} are non negative. Hence, we obtain that $M^{n} = \tilde M^{n}$ for all $0 \le n \le N$. By Prop. \ref{prop:ex_HJ}, it also follows   that $U^{n} = \tilde U^{n}$ for all $0 \le n \le N$.
\end{proof}

\subsection{Existence}
To show the existence of a solution for the scheme \eqref{eq:scheme}, we rely on a fixed  point argument. 
\begin{theorem}
For $\Dt$ sufficiently small,  \eqref{eq:scheme} has a solution.
\end{theorem}
\begin{proof}
First observe that, by the assumption on $m_0$, we have that $M^0\in\cK_h$. We denote    
\begin{equation}\label{eq:def_K_N}
 \cK_h^N=\{M=(M^n)_{0\le n\le N} :\, M^0= \frac 1 {h^2} \int_{|x-x_{i,j}|_{\infty}\le h/2} m_0(x) dx,\,  M^n\in\cK_h,\,\forall n=0,\dots,N\} 
\end{equation}
and we define a map $\Phi$ in the following way:  
Let $\Psi_1:\cK^N_h\to (\R^{N_h^{2}})^N$ be the map which associates to $M=(M^n)_{0\le n\le N}$  the solution 
$U=(U^n)_{0\le n\le N}$ of \eqref{discrete_HJequation} and let $\Psi_2: (\R^{N_h^2})^N\to (\R^{N_h^2})^N$ the map that    to $U=(U^n)_{0\le n\le N}$ associates the solution of \eqref{discrete_FPequation}. Then we show that the map  $\Phi=\Psi_2\circ \Psi_1$, is well defined on  $\cK^N_h$ and maps this set into itself.\\

\indent We show that the map $\Psi_1$ is bounded by means of a barrier argument. By \eqref{eq:barrier1}, the function
$$V^n=\|U^N\|_\infty + \frac{M_0\Gamma (2-\alpha)}{\alpha(1-\alpha)}((N-n)\Dt)^\alpha,$$
where $M_0=\sup_{i,j}|H(x_{i,j},0)-(f_h(M^{n+1}))_{i,j}|$, satisfies $V^N\ge U^N$ and
$$\bar D_{\Delta t}^{\alpha} V_{i,j}^{n} - (\Delta_h V^{n})_{i,j} + g(x_{i,j},[D_hV^{n}]_{i,j}) \ge  f_h[M^{n+1}]_{i,j},\quad n=0,\dots,N-1.$$
By Proposition \ref{prop:ex_HJ}, we get $U^n\le V^n$, $n=0,\dots,N$. In a similar way, we prove that
$U^n\ge- V^n$, $n=0,\dots,N$, and therefore
$$\|U^n\|_\infty \le \|U^N\|_\infty +CT^\alpha.$$
Hence $\Psi_1$ maps $\cK_h^N$ in a bounded subset of $(\R^{N_h^2})^N$.\\
\indent To prove that the map $\Psi_1$ is continuous, consider a sequence $M_{\kappa}=(M_{\kappa}^n)_{0\le n\le N}\in \cK_h^N$ which tends to $M=(M^n)_{0\le n\le N}\in \cK_h^N$ for $\kappa \to \infty$ and denote by $U_{\kappa}$, $U \in (\R^{N_h^2})^N$ the  corresponding solutions of \eqref{discrete_HJequation}. Set 
\[
E_{\kappa} = \sup_{ n=0,\dots,N-1}  \left\| f_h[M_{\kappa}^{n+1}]  -f_h[M^{n+1}] \right\|_\infty
\]
and observe that, by  assumption (F1) on $f_h$, $E_{\kappa}\to 0$ for $k\to \infty$. Set $V^n_{\kappa}=U^n_{\kappa}- \frac{\Gamma(2-\alpha)}{\alpha(1-\alpha)}((N-n)\Dt)^\alpha E_{\kappa}$, $n=0,\dots,N$, then 
\begin{align*}
\bar D_{\Delta t}^{\alpha} V^n_{i,j}  - (\Delta_h V^{n})_{i,j} + g(x_{i,j},[D_h V^{n}]_{i,j})\le  f_h[M_k^{n+1}]_{i,j}-E_{\kappa}= \\
f_h[M^{n+1}]_{i,j}+f_h[M_k^{n+1}]_{i,j}-f_h[M^{n+1}]_{i,j}-  E_{\kappa}
\le f_h[M^{n+1}]_{i,j}.
\end{align*}
\indent Moreover $V^N\le U^N$ and therefore, by Prop. \ref{prop:ex_HJ}, $V^n_{\kappa}\le U^n$ for any $n=0,\dots,N$. Using a similar argument for $W^n_{\kappa}=U^n_{\kappa}+\frac{\Gamma(2-\alpha)}{\alpha(1-\alpha)}((N-n)\Dt)^\alpha E_{\kappa}$, $n=0,\dots,N$, we   get that for any $\kappa$
\[\sup_{n=0,\dots,N}\|U^n_{\kappa} -U^n\|_\infty\le \frac{\Gamma(2-\alpha)}{\alpha(1-\alpha)}T^\alpha	 E_{\kappa}\]
and therefore the continuity of the map $\Psi_1$.\\
\indent Given $U=\Psi_1(M)$, we first show that, for $\Dt$ sufficiently small, $\tilde M=\Psi_2(U)$ is well defined. First observe that, by Prop. \ref{prop:lip_U}, $\|D_h U\|_\infty$ is bounded uniformly in $M$. Therefore, since $g$ is $C^1$, there exists a constant $C$, independent of $M$, such that
for $k=1,2,3,4,$
\begin{equation}\label{eq:Fix1}
\left|\frac{\partial g}{\partial q_k} (x_{i,j},[D_hU^n]_{i,j})\right|\le C, \quad 0\le i,j\le N_h,\,0\le n\le N.
\end{equation}
\indent Hence, arguing as in Prop. \ref{prop:ex_FP}, we obtain that, for $\Dt$ sufficiently small, there exists a solution $\tilde M$ of \eqref{discrete_FPequation}. Moreover $\tilde M\in \cK_h^{N}$. The continuity of the map $\Psi_2$ follows by the  regularity of $g$ and the linearity of the problem.\\
\indent The boundedness and  continuity of the map $\Phi_1$ and the continuity of the map $\Psi_2$ implies that   $\Phi=\Psi_2\circ\Psi_1$ is a continuous map on the compact, convex set $\cK_h^N$. Moreover by Prop. \ref{prop:ex_FP}, $\Phi$ maps $\cK_h^N$ into itself. Hence, by Brouwer's Theorem, $\Phi$ admits a fixed point $M$, which clearly a solution of system \eqref{eq:scheme}.
\end{proof}


\section{Convergence}\label{sec:convergence}
In this section we prove a convergence result for the scheme \eqref{eq:scheme}. We relies on the analysis performed in \cite[Section 3]{acc} for the Mean Field Games system and in \cite[Theorem 3.2]{lx} for computation related to fractional derivatives. \\
\indent A preliminary result for the  analysis is a fundamental identity which relies on the duality structure of the discrete system.
\begin{proposition}
Let  $(U,M)$ be   a solution of \eqref{eq:scheme} and  $(\tdU,\tdM)$ a solution of the perturbed system
\begin{equation}\label{eq:conv_scheme_pert}
\left\{
\begin{array}{ll}
\bar D_{\Delta t}^{\alpha} \tdU_{i,j}^{n} - (\Delta_h \tdU^{n})_{i,j} + g(x_{i,j},[D_h\tdU^{n}]_{i,j}) = f_h[\tdM^{n+1}]_{i,j}+a^n_{i,j}, \\[6pt]
D_{\Delta t}^\alpha \tdM_{i,j}^{n+1} - (\Delta_h \tdM^{n+1})_{i,j} - \bB_{i,j}(\tdU^{n},\tdM^{n+1}) = b^{n+1}_{i,j}, \\[6pt]
\tdU^{N}_{i,j} = u_T(x_{i,j}), \,   \tdM_{i,j}^{0}=  \frac 1 {h^2} \int_{|x-x_{i,j}|_{\infty}\le h/2} m_0(x) dx,
\end{array}
\right.
\end{equation}
for $n=0,\dots,N-1$ and $0\le i,j\le N_h$.	 Then
\begin{equation}\label{eq:conv_fund_id}
\begin{aligned}
&  \sum_{n=0}^{N-1}   \bigg(f_h[M^{n+1}] - f_h[\tilde{M}^{n+1}],M^{n+1} - \tilde{M}^{n+1}\bigg)_2 \\
& +   \sum_{n=0}^{N-1}  \sum_{i,j} M_{i,j}^{n+1}R^n_{i,j}(U,\tdU)+ \sum_{n=0}^{N-1}  \sum_{i,j} \tilde{M}_{i,j}^{n+1} R^n_{i,j}(\tdU,U)\\
&=\sum_{n=0}^{N-1}(a^n,M^{n+1}-\tdM^{n+1})_2+\sum_{n=0}^{N-1}(b^{n+1},U^n-\tdU^n)_2
\end{aligned}
\end{equation}
where
\[
R^n_{i,j}(V,W)=g(x_{i,j},[D_hW^{n}]_{i,j}) - g(x_{i,j},[D_hV^{n}]_{i,j})   -[D_h(W^{n}-V^{n})]_{i,j} \cdot  g_q  (x_{i,j},[D_hV^{n}])
\]
for $V=(V^n)_{0\le n\le N}$, $W=(W^n)_{0\le n\le N}$.
\end{proposition}
\begin{proof}
The identity \eqref{eq:conv_fund_id} can be obtained by arguing as in the proof of Theorem \ref{thm:uniqueness}, see in particular \eqref{eq:un_nonneg_terms}, and taking into account the additional terms $a^n_{i,j}$, $b^n_{i,j}$.
\end{proof}

For the convergence analysis, we consider the reference case
\begin{equation}\label{eq:conv_ham}
H(x,p)=|p|^\beta+c(x)
\end{equation}
where $\beta\ge 2$ and $c:\T^2\to\R$ is a continuous function, but the analysis can be extended in a more general setting. For the approximation of the Hamiltonian, we  consider the  numerical Hamiltonian  
\begin{equation}\label{eq:conv_num_ham}
g(x,q)=((q_1^-)^2+(q_2^+)^2+(q_3^-)^2+(q_4^+)^2)^{\beta/2}+c(x)
\end{equation}
where    $q^\pm$ denote the positive and the negative part of $q\in\R$. 
\begin{theorem}
	Assume   that the continuous problem \eqref{eq:frac_mfg} admits a unique classical solution $(u,m)\in C^{2}(\T^2\times [0,T])\times C^{2}(\T^2\times [0,T])$ and that $m_0\ge \delta>0$. For $\s=(h,\Dt)$, denote by $(u_\s,m_\s)$ the linear function obtained by interpolating the solution $(U,M)$ of \eqref{eq:scheme} on the nodes of the space-time grids. Then, as  $|\s|\to 0$, $(u_\s,m_\s)$ converges to $(u,m)$ in $L^\beta(0,T: W^{1,\beta}(\T^2))\times  (C(0,T;L^2(\T^2)) \cap L^2(0,T;H^1(\T^2))$. 
	\end{theorem}
\begin{proof}
The idea of the proof  is to show that the solution of the continuous problem \eqref{eq:frac_mfg} satisfies the perturbed discrete  problem \eqref{eq:conv_scheme_pert}  with $(a^n,b^n)$ vanishing for the discretion steps $\Dt,h$ going to $0$.\\
\indent Given a solution   $(u,m)$ of \eqref{eq:frac_mfg}, we define two  grid functions $\tdU_\s$,$\tdM_\s$ on $(\R^{N_h^2})^N$ by
\[(\tdU_\s)^n_{i,j}=u(x_{i,j},n\Dt),\quad (\tdM_\s)^n_{i,j}=\frac 1 {h^2} \int_{|x-x_{i,j}|\le h/2} m(x,n\Dt) dx .\]
\indent Note that $\tdM_\s\in \cK_h^N$, see \eqref{eq:def_K_N}. Moreover, since $(u,m)$ is a smooth function, 
$(\tdU_\s,\tdM_\s)$ is a solution of \eqref{eq:conv_scheme_pert} with the perturbation $(a^n_\s,b^n_\s)$ given by the consistency error of the scheme. Hence
\begin{equation}\label{eq:conv_consistency}
\max_{0\le n\le N}\{\|a^n_\s\|_\infty+\|b^n_\s\|_\infty\}\to 0 \quad \text{for $\s\to 0$}.
\end{equation}
\indent The first step of the proof is  the convergence result
\begin{equation}\label{eq:conv_seminorm}
\lim_{|\s|\to 0}h^2\Dt\sum_{n=0}^N\sum_{i,j}|[D_hU^n_\s]_{i,j}-[D_h(\tdU_\s)^n]_{i,j}|^\beta=0.
\end{equation} 
The identity \eqref{eq:conv_seminorm}  is obtained by means of the fundamental identity \eqref{eq:conv_fund_id} which does not involve the discrete fractional derivatives. Hence, for its proof, it is possible to repeat  the same proof of Step 1 in \cite[Section 3]{acc}. Note that here it is employed that $\tdM_\s>0$, which follows by the assumption $m_0\ge\delta>0$. \\
\indent The second step of the proof is the convergence of the solution of the Fokker-Planck equation. We denote with $M$, $\tdM$ the   grid functions corresponding to $m_\s$ and $m$, omitting the index $\s$ for simplicity. We set
$E=M-\tdM$ and 
\begin{align*}
&\|[D_hE^{n+1}]\|^2_2=h^2\sum_{i,j} [D_hE^{n+1}]_{i,j}\cdot[D_hE^{n+1}]_{i,j},\\
&\|(E^{n+1},[D_hE^{n+1}])\|_2^2=\|E^{n+1}\|^2_2+\rho_\alpha\|[D_hE^{n+1}]\|^2_2.
\end{align*} 
By subtracting the equations satisfied  by $M$, $\tdM$, we get
\[
E^{n+1}_{i,j}-\rho_\alpha (\Delta_h E^{n+1})_{i,j} - \rho_\alpha(\bB_{i,j}(U^{n},M^{n+1})-\bB_{i,j}(\tdU^{n},\tdM^{n+1})) =\sum_{k=0}^n c^n_k E^k+\ra b^{n+1}.
\]
Multiplying the previous equation by $E^{n+1}_{i,j}$, summing over $i$ and $j$ and using the identities \eqref{integralterm} and  \eqref{eq:bypart_space}, we get
\begin{equation}\label{eq:conv_es1}
\begin{split}
&\|(E^{n+1},[D_hE^{n+1}])\|_2^2-\ra h^2 \sum_{i,j} E^{n+1}_{i,j} [D_hE^{n+1}]_{i,j}\cdot g_q(x_{i,j}, [D U^n]_{i,j})\\
&=h^2\sum_{k=0}^n c^n_k (E^k,E^{n+1})+\ra h^2(b^{n+1}, E^{n+1})\\
&+\ra h^2  \sum_{i,j}\tdM_{i,j}^{n+1}[DE^{n+1}]_{i,j}\cdot(g_q(x_{i,j}, [D_hU^n]_{i,j})-g_q(x_{i,j}, [D_h\tdU^n]_{i,j})).
\end{split}
\end{equation}
Taking into account \eqref{eq:lip_est}, we estimate for $\Dt$ sufficiently small
\[
\left|\ra h^2 \sum_{i,j} E^{n+1}_{i,j} [D_hE^{n+1}]_{i,j}\cdot g_q(x_{i,j}, [D U^n]_{i,j})\right|\le 
\frac 1 2\|(E^{n+1},[D_hE^{n+1}])\|_2^2.
\]
Using again  \eqref{eq:lip_est} and since $(u,m)$ is a classical solution to \eqref{eq:frac_mfg},   we have for all $\eta>0$ (see\cite[Equation (4.5)]{acc}) 
\begin{align*}
&\left|\tdM_{i,j}^{n+1}[DE^{n+1}]_{i,j}\cdot(g_q(x_{i,j}, [D_hU^n]_{i,j})-g_q(x_{i,j}, [D_h\tdU^n]_{i,j}))\right|\\
&\le  C \Big(\frac{1}{\eta} |[D_hE^{n+1}]_{i,j}|^2+\eta |[D_hU^n]_{i,j}-[D_h\tdU^n]_{i,j}|^\beta\Big)
\end{align*}
Plugging the previous estimates in \eqref{eq:conv_es1} and dividing by $\|(E^{n+1},[D_hE^{n+1}])\|_2^2$, we get
\begin{equation}\label{eq:conv_es2}
\|(E^{n+1},[D_hE^{n+1}])\|_2\le C\left(\sum_{k=0}^n c^n_k \|E^k\|_2+\ra\varepsilon^{n+1}\right)
\end{equation}
where
\[
\varepsilon^{n+1}=\|b^{n+1}\|_\infty+  h^2  \sum_{i,j}|[D_hU^n]_{i,j}-[D_h\tdU^n]_{i,j}|^\beta.
\]
and the constant $C$ is independent of $h,\Dt$.\\ 
\indent Following \cite[Theorem 3.2]{lx}, we rewrite \eqref{eq:conv_es2} as
\begin{equation}\label{eq:conv_es3}
\|(E^{n+1},[D_hE^{n+1}])\|_2\le C\Big((1-B_1)\|E^n\|_2+\sum_{k=1}^{n-1}(B_k-B_{k+1})\|E^{n-k}\|_2
+B_n\|E^0\|_2+\ra\varepsilon^{n+1}\Big),
\end{equation}
where $B_k=(k+1)^{1-\alpha}-k^{1-\alpha}$, $k=0,\dots,n$. 
Recall that,  since $M^0=\tdM^0$, then $E^0\equiv 0$. We claim that for $n=1,\dots,N$,
\begin{equation}\label{eq:conv_es4}
\|(E^{n},[D_hE^{n}])\|_2\le C\ra B_{n-1}\max_{0\le k\le N}\varepsilon^k.
\end{equation}
For $n=1$, by \eqref{eq:conv_es3} we have 
\[\|(E^{1},[D_hE^{1}])\|_2\le C\ra\varepsilon^1\le C \ra B_0^{-1}\max_{0\le k\le N}\varepsilon^k, \]
since $B_0=1$. Assuming that \eqref{eq:conv_es4} is true for $n$, by \eqref{eq:conv_es3} and the fact that
$B_k^{-1}/B_{k+1}^{-1}<1$, we obtain
\begin{align*}
&\|(E^{n+1},[D_hE^{n+1}])\|_2\le \left[ (1-B_1)B_k^{-1}+\sum_{k=1}^{n-1}(B_k-B_{k+1})B^{-1}_{n-k-1}\right]C\ra\max_{0\le k\le N}\varepsilon^k\\
&+C\ra\max_{0\le k\le N}\varepsilon^k
\le \left[ (1-B_1) +\sum_{k=1}^{n-1}(B_k-B_{k+1})+B_n\right]C  \ra B_n^{-1}\max_{0\le k\le N}\varepsilon^k.
\end{align*}
Using $(1-B_1) +\sum_{k=1}^{n-1}(B_k-B_{k+1})+B_n=1$ (see \cite[Equation (3.7)]{lx}),  the claim 
\eqref{eq:conv_es4} follows.\\
Since $n^{-\alpha}B_{n-1}^{-1}\le 1/(1-\alpha)$ for $n=1,\dots,N$ (see   \cite[Equation (3.21)]{lx}) and    $\ra=\Gamma(2-\alpha)\Delta t^{\alpha}$, by \eqref{eq:conv_es4} we have
\begin{equation}\label{eq:conv_es5}
\|(E^{n},[D_hE^{n}])\|_2\le C\max_{0\le k\le N}\varepsilon^k.
\end{equation}
\indent Recalling that $E^n$ and $\varepsilon^n$ depends on $\s$ and that, by \eqref{eq:conv_consistency}-\eqref{eq:conv_seminorm},
$\lim_{|\s|\to 0}\max_{0\le k\le N}\varepsilon^k_\s=0$, we finally get
\[
\lim_{|\s|\to 0}\,\max_{0\le n\le N}\{\|E_\s^{n}\|_2+\rho_\alpha\|[D_hE_\s^{n}]\|_2\}=0,
\]
which gives the convergence of $m_\s$ to $m$ in $C(0,T;L^2(\T^2))$. To get the convergence in $L^2(0,T;H^1(\T^2))$, it is sufficient to observe that, since
$m$ is smooth, the consistency error of the approximation of the fractional derivative is of order $\Dt^{2-\alpha}$ (see \cite[Equation (3.3)]{lx}). Hence, by \eqref{eq:conv_es5},
it follows $\lim_{|\s|\to 0} \Dt\sum_{n=1}^N\|[D_hE_s^{n}]\|_2=0$ and therefore the convergence of the gradient of the error.\\ 
\indent In the final step of the proof,  we prove the convergence of $u_\s$ to $u$. By the convergence
of $m_\s$ to $m$ and assumption  (F4), we have that $\tdU$ is a solution to a discrete Hamilton-Jacobi-Bellman
equation with $M$ in place of $\tdM$, i.e.
\begin{equation}\label{eq:conv_HJ_pert}
\left\{
\begin{array}{ll}
\bar D_{\Delta t}^{\alpha} \tdU_{i,j}^{n} - (\Delta_h \tdU^{n})_{i,j} + g(x_{i,j},[D_h\tdU^{n}]_{i,j}) = f_h[M^{n+1}]_{i,j}+a^n_{i,j},\\[6pt] 
\tdU^{N_T}_{i,j} = u_T(x_{i,j}),
\end{array}
\right.
\end{equation}
where the perturbation $a^n$ still satisfies \eqref{eq:conv_consistency}. Arguing as in Prop. \ref{prop:lip_U} and with the same notation, we have
$U^n= \Psi\left(\sum_{k=n+1}^{N} \bc_{n}^k U^k\right)$, $\tdU^n= \Psi\left(\sum_{k=n+1}^{N} \bc_{n}^k \tdU^k+a^n\right)$. Hence
\begin{align*}
&\|\tdU^n-U^n\|_\infty= \left\|\Psi\left(\sum_{k=n+1}^{N} \bc_{n}^k \tdU^k+\rho_\alpha a^n\right)
-\Psi\left(\sum_{k=n+1}^{N} \bc_{n}^k  U^k\right)\right\|_\infty\\
&\le \sum_{k=n+1}^{N}\bc_{n}^k \left\| \tdU^k - U^k\right\|_\infty+\ra\|a^n\|_\infty.
\end{align*} 
The previous inequality is equivalent to
\[
\bar D_{\Delta t}^{\alpha}(\|\tdU^n-U^n\|_\infty)\le \|a^n\|_\infty.
\]
We conclude, by the backward version of discrete Gronwall  inequality \eqref{eq:discrete_Gronwall}, the
uniform convergence of $u_\s$ to $u$. The convergence of the gradients follows by \eqref{eq:conv_seminorm}.
\end{proof}

\begin{remark}
In this paper we have focused on the numerical approximation of the classical solution of the PDE system \eqref{eq:frac_mfg} and we have assumed the existence of a unique classical solution. We refer to \cite{q} for the existence and uniqueness of a weak solution to the system, where the solution of the fractional Fokker-Planck equation is defined in a time-fractional Bochner space and the equation is considered in the sense of distributions.  It would be very interesting to define some notion of renormalized solution to fractional PDEs and consider weak solution to the system, as well as approximation schemes. In the standard case,  convergence of a finite differences scheme to the weak solution of  a Mean Field Games system was shown in \cite{ap}.
\end{remark}

\section{Numerical simulations}\label{sec:numerics}
In this section, we illustrate some numerical simulations in one-dimensional case. In all  examples, we consider the domain $\T = [0,1]$ and the final time $T=2$. We test
\begin{equation}\label{eq:test}
\begin{cases}
\partial_{[t,T)}^\alpha u -\sigma \Delta u + \frac{u_x^2}{2} = F(x,t,m) \\[4pt]
\partial_{(0,t]}^\alpha m - \sigma \Delta m - (mu_x)_x = 0 \\[4pt]
\end{cases}
\end{equation}
with terminal-initial data $u(x,T)=0$ and $m(x,0)=\nu(x)/\int_\T \nu dx$
where $\nu(x) = e^{-(x-0.5)^2/(0.1)^2}$. Here, the running cost   function has a form
\[
F(x,t,m) = 5\left(x-\frac{1-\sin (2 \pi t)}{2}\right)^2 + \lambda m
\]
where $\lambda$ is a parameter.

\textbf{Test 1.}
For $(x,t)\in\T\times [0,2]$, we  take $\sigma = 0$ and $\lambda = 0$, i.e.  there is no diffusion and penalization term in \eqref{eq:test}. Even if this case is not covered by the results of the paper,  we consider it to highlight the effects of the fractional derivative. For the running cost in the MFG system, the agents try to concentrate at the point $(1-\sin (2 \pi t))/2$. We apply \eqref{eq:scheme} via full discretization of \eqref{eq:test}. We take $N_h=50$ and $N=2000$ so that $h=\frac{1}{50}$ and $\Delta t = \frac{T}{2000}$. Since in this case the Hamilton-Jacobi-Bellman equation in the system \eqref{eq:test} does not depend on $m$, the system is decoupled. In the numerical computation, we start from the terminal data $U_i^{N_T} = u_T(x_i)$ and go backward to iteratively calculate $U^n$ on the whole grid for all $0 \le n < N_T$. Having the numerical solution for the Hamilton-Jacobi-Bellman equation, we now can compute $M^n$ on the grid for all $0<n \le N$ starting from the initial time. In Figure 1, we provide density evolution for three different values of $\alpha$. From the level sets in Figure 2, we can see that the fractional time derivative has a kind of diffusion effect. Also, from Figure 2, we notice that when $\alpha$ gets smaller, the shape of the sinusoidal curve becomes more and more asymmetric and smoother. This means that for the smaller values of $\alpha$, in the mean-field game, the agents try to spread away from the sinusoidal curve.

\begin{center}
\begin{tabular}{ccc}
	\includegraphics[width=50mm]{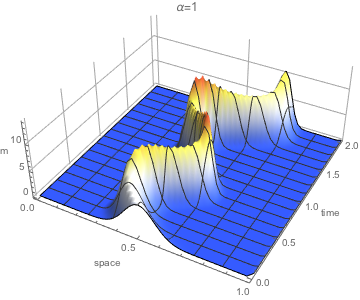}&		\includegraphics[width=50mm]{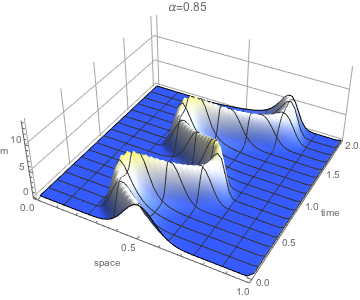}&
	\includegraphics[width=50mm]{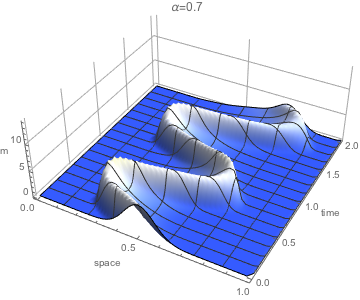}\\
	(A)&(B)&(C)	
\end{tabular}\\
	\captionof{figure}{Mass evolution (A) for $\alpha=1$ (B) for $\alpha=0.85$ (C) for $\alpha=0.7$.}
\end{center}

\begin{center}
	\begin{tabular}{ccc}
	\includegraphics[width=50mm]{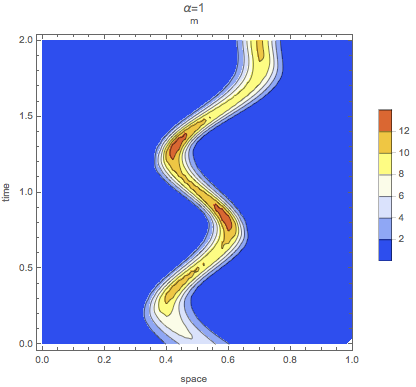}	&
	\includegraphics[width=50mm]{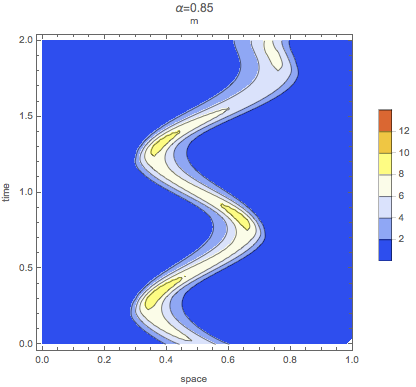}&
	\includegraphics[width=50mm]{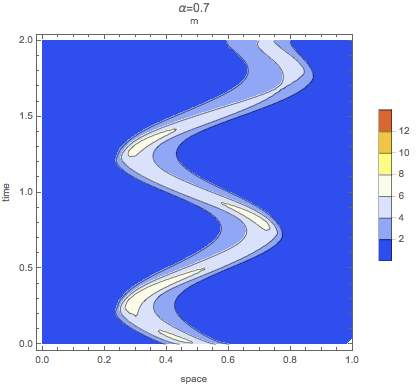}\\
	(A)&(B)&(C)	
\end{tabular}\\
	\captionof{figure}{Level sets of the mass evolution (A) for $\alpha=1$ (B) for $\alpha=0.85$ (C) for $\alpha=0.7$.}
\end{center}

\textbf{Test 2.} Now, we consider \eqref{eq:test} with $\sigma = 0.1$ and $\lambda=0$ to see how the diffusion term affects the density evolution. For this case, we provide only level sets of the density in Figure 3. During the whole time interval, a diffusive effect is observed. Moreover, its effect is much stronger than the fractional time derivative effect which is displayed in Figure 2. For example, we can compare Figure 2 (B), where $\alpha=0.85, \sigma =0$ and Figure 3 (A), where $\alpha=1, \sigma =0.1$.

     \begin{center}
	\begin{tabular}{ccc}
	\includegraphics[width=50mm]{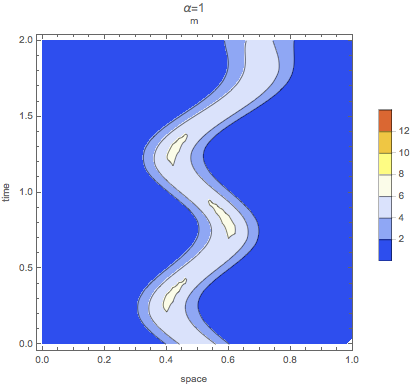}	&
	\includegraphics[width=50mm]{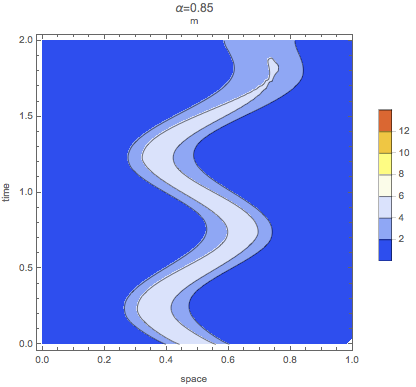}&
	\includegraphics[width=50mm]{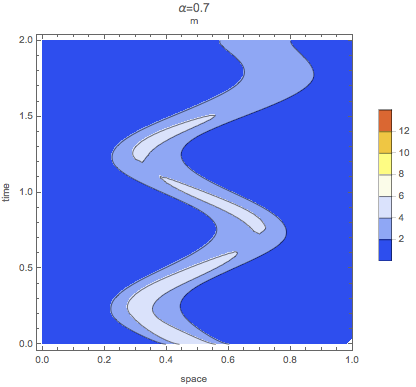}\\
	(A)&(B)&(C)	
	\end{tabular}\\
\captionof{figure}{Level sets of the mass evolution (A) for $\alpha=1$ (B) for $\alpha=0.85$ (C) for $\alpha=0.7$.}
	\end{center}

\textbf{Test 3.} We test the case where $\sigma=0$ and $\lambda=1$. Since the coupling cost is increasing in $m$, it penalizes concentration of the agents. Indeed, we see from Figure 4  that concentration of agents is not as high as in Test 1.

\begin{center}
	\begin{tabular}{ccc}
	\includegraphics[width=50mm]{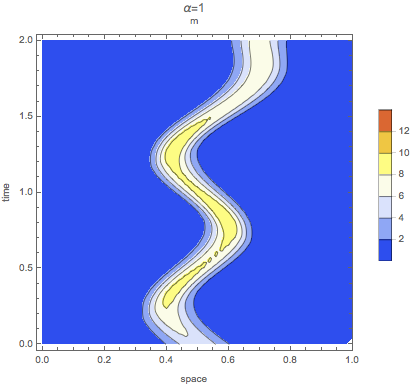}	&
	\includegraphics[width=50mm]{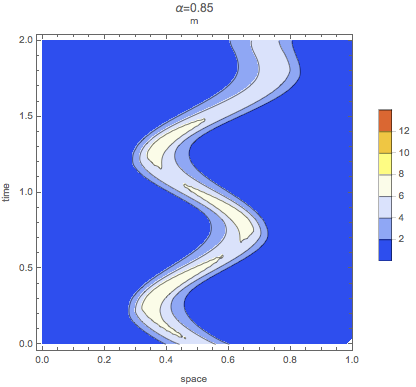}&
	\includegraphics[width=50mm]{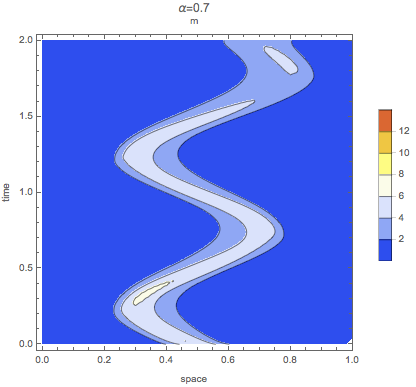}
	\\
	(A)&(B)&(C)	
\end{tabular}\\
	\captionof{figure}{Level sets of the mass evolution (A) for $\alpha=1$ (B) for $\alpha=0.85$ (C) for $\alpha=0.7$.}
\end{center}

\section{Discussion}
\indent In this paper, we considered the numerical approximation to a time-fractional forward backward PDE system. By using a classical scheme from Mean Field Games theory and L1 scheme for fractional derivatives the duality structure of the PDE system is preserved on the discrete level. The duality with regard to the fractional derivatives is further clarified in the following discrete fractional integration by parts formula in the appendix. In future work we may consider the error estimates of approximation scheme to the viscosity solution of the fractional Hamilton-Jacobi-Bellman equation as in \cite{glm}. Another interesting direction will be to consider solving the fractional forward backward system by machine learning methods. Machine learning methods for solving mean field games have been considered by Lauriere and Carmona in \cite{ml}. In \cite{Pang}, Pang, Li and Karniadakis used machine learning methods to solve forward and inverse problems of time-space fractional PDEs. A combination of techniques might produce fruitful results for studying fractional PDE systems. 

\appendix
\section{Appendix}
\begin{proof}[Proof of Lemma \ref{le:by_parts_n_n_plus_1}]
In this proof, we use (ii) in Lemma \ref{lem:c_properties}.
We have
\[
\begin{aligned}
\ra & \sum_{n=0}^{N-1} \left( \bar D_{\Delta t}^{\alpha} U^{n}, M^{n+1} \right)_2 + \sum_{n=0}^{N-1} \bar c_n^N \left(U^N,M^{n+1}\right)_2  \\
& = \sum_{n=0}^{N-1} \left(U^n - \sum_{k=n+1}^{N} \bar c_n^k U^k,M^{n+1}\right)_2 + \sum_{n=0}^{N-1} \bar c_n^N \left(U^N,M^{n+1}\right)_2 \\
& = \sum_{n=0}^{N-1} (U^n,M^{n+1})_2 - \sum_{n=0}^{N-2} \sum_{k=n+1}^{N-1} \left( \bar c_n^k U^k,M^{n+1}\right)_2 \quad \text{(change the order of sums)}\\
& =  \sum_{n=0}^{N-1} (U^n,M^{n+1})_2 - \sum_{k=1}^{N-1} \sum_{n=0}^{k-1} \left( \bar c_n^k U^k,M^{n+1}\right)_2 \quad\quad \text{(interchange $m$ and $n$)}\\
& = \sum_{n=0}^{N-1} (U^n,M^{n+1})_2 - \sum_{n=1}^{N-1} \sum_{k=0}^{n-1} \left( \bar c_k^n U^n,M^{k+1}\right)_2\\
& = (U^0,M^1)_2 + \sum_{n=1}^{N-1} (U^n, M^{n+1} - \sum_{k=0}^{n-1} \bar c_{k}^{n} M^{k+1})_2 \\
& = (U^0,M^1)_2 + \sum_{n=1}^{N-1} (U^n, M^{n+1} - \sum_{k=0}^{n-1} c_{k+1}^{n+1} M^{k+1})_2 \quad\quad  \text{(since $\bar c_{k}^{n} = c_{k+1}^{n+1}$)}\\
& = (U^0,M^1)_2 + \sum_{n=1}^{N-1} (U^n, M^{n+1} - \sum_{k=1}^{n} c_{k}^{n+1} M^{k})_2 \\
& = (U^0,M^1 - c_0^1 M^0)_2 + \sum_{n=1}^{N-1} (U^n, M^{n+1} - \sum_{k=0}^{n} c_{k}^{n+1} M^{k})+\sum_{n=0}^{N-1} c_0^{n+1}(U^n,M^0)_2 \\
& = \sum_{n=0}^{N-1} (U^n, M^{n+1} - \sum_{k=0}^{n} c_{k}^{n+1} M^{k})_2 + \sum_{n=0}^{N-1} c_0^{n+1}(U^n,M^0)_2 \\
& = \ra \sum_{n=0}^{N-1} (D_{\Delta t}^{\alpha} M^{n+1},U^{n})_2 + \sum_{n=0}^{N-1} c_0^{n+1}(U^n,M^0)_2.
\end{aligned}
\]
Dividing by $\ra$, we obtain the result.
\end{proof}

\end{document}